\documentclass[article, 12pt]{amsart}

\usepackage{enumerate,amsmath,amssymb, fullpage}
\usepackage{hyperref}
\usepackage{yhmath}
\usepackage{color}
\usepackage{graphicx, times}

\theoremstyle{plain}

\makeatletter
\newtheorem*{rep@theorem}{\rep@title}
\newcommand{\newreptheorem}[2]{%
\newenvironment{rep#1}[1]{%
\def\rep@title{#2 \ref{##1}}%
\begin{rep@theorem}}%
{\end{rep@theorem}}}
\makeatother

\makeatletter
\def\oversortoftilde#1{\mathop{\vbox{\m@th\ialign{##\crcr\noalign{\kern3\p@}%
      \sortoftildefill\crcr\noalign{\kern3\p@\nointerlineskip}%
      $\hfil\displaystyle{#1}\hfil$\crcr}}}\limits}

\def\sortoftildefill{$\m@th \setbox\z@\hbox{$\braceld$}%
  \braceld\leaders\vrule \@height\ht\z@ \@depth\z@\hfill\braceru$}

\makeatother

\newtheorem{thm}{Theorem}[section]
\newreptheorem{theorem}{Theorem}

\newtheorem{lemma}[thm]{Lemma}

\theoremstyle{definition}

\newtheorem{definition}[thm]{Definition}
\newtheorem{example}[thm]{Example}

\renewcommand{\epsilon}{\varepsilon}

\let\theta\vartheta
\let\phi\varphi

\let\<\langle
\let\>\rangle

\DeclareMathAlphabet{\doba}{U}{msb}{m}{n}

\newcommand{\bq}{\begin{equation}}
\newcommand{\eq}{\end{equation}}

\newcommand{\definedas}{\mathrel{\raise.095ex\hbox{\rm :}\mkern-5.2mu=}}

\begin{document}
%

\title
[]
{Hybrid functions approach to solve a class of Fredholm and Volterra Integro-Differential equations}

\author{Aline Hosry, Roger Nakad and Sachin Bhalekar}

\address{Aline Hosry \\
Lebanese University \\
Faculty of Science II\\
Department of Mathematics\\
P. O. Box 90656, Fanar, El Metn\\
Lebanon}
\email{aline.hosry@ul.edu.lb}

\address{Roger Nakad \\
Notre Dame University-Louaiz\'e \\
Faculty of Natural and Applied Sciences\\
Department of Mathematics and Statistics\\
P.O. Box 72, Zouk Mikael\\
Lebanon}
\email{rnakad@ndu.edu.lb}

\address{Sachin Bhalekar \\
Shivaji University\\
Department of Mathematics\\
Vidyanagar, Kolhapur, 416 004\\
India}
\email{\texttt{sbb\_maths}@unishivaji.ac.in}

\thanks{The first two authors are listed by alphabetical order of their family names.}

\date{\today}



\begin{abstract} 
In this paper, we use a numerical method that involves hybrid and block-pulse functions to  approximate solutions of systems of a class of Fredholm and Volterra integro-differential equations. The key point is to derive a new approximation for the derivatives of the solutions and then  reduce the integro-differential equation to a system of algebraic equations that can be solved using classical methods. Some numerical examples are dedicated for showing efficiency and validity of the method that we introduce.

\end{abstract}
\maketitle
\tableofcontents

\section{Introduction}

Scientific researchers have explored the topic of integro-differential equations through their work in various fields of science such as physics \cite{phs}, biology \cite{bio1} and engineering \cite{eng1, eng2} and in numerous applications 
such as heat transfer, neurosciences \cite{neuro}, diffusion process, neutron diffusion, biological species \cite{Basirat-Malek-Hashemi, neutron}, biomechanics, economics, electrical engineering, electrodynamics, electrostatics, filtration theory, fluid dynamics, game theory, oscillation theory, queuing theory \cite{qu}, airfoil theory \cite{ar}, elastic contact problems \cite{el1, el2}, fracture mechanics \cite{fm}, combined infrared radiation and molecular conduction \cite{mol} and so on.\\

In recent years, many different basic functions have been used to estimate the solution of integral equations, such as orthogonal functions and wavelets. Three families of the orthogonal functions are classified: piecewise constant orthogonal functions (e.g., Walsh, Haar, block-pulse, etc.), orthogonal polynomials (e.g., Legendre, Laguerre, Chebyshev, etc.) and sine-cosine functions in the Fourier series. For instance, many authors investigated the general $k^{th}$ order integro-differential equation 

\begin{eqnarray}\label{genequation}
y^{(k)}(t) + l(t) y(t) + \int_a^b g(t, s )y^{(m)}(s) ds = f(t),
\end{eqnarray}
with initial conditions $$y(a) = a_0, \dots, y^{(n-1)}(a) = a_{n-1},$$
where $a_0, \ldots, a_{n-1}$ are real constants, $k, m $ are positive integers and $m< k$, the functions $l, f, g$ are given and $y(t)$ is the solution to be determined.  In \cite{HeHom}, the authors applied the homotopy perturbation method to solve Equation (\ref{genequation}), while in \cite{VIM2, VIM1}, the authors changed the equation to an ordinary integro-differential equation and applied the variational iteration method to solve it so that the Lagrange multipliers can be effectively identified. Using the operational matrix of derivatives of hybrid functions, a numerical method  has been presented in \cite{Hybgen} to solve Equation (\ref{genequation}). In \cite{hemada}, Hemeda used the iterative method introduced in \cite{NIM} to solve the more general  equation 
\begin{eqnarray}\label{genequation2}
y^{(k)}(t) + l(t) y(t) + \int_a^b g(t, s )y^{(n)}(s)y^{(m)}(s) ds = f(t),
\end{eqnarray}
where $n \leq m <q$. \\ 

In this paper, we use block-pulse and hybrid functions to approximate solutions $y(t)$ of the Fredholm integro-differential system given by
\begin{eqnarray}
\left \{
\begin{array}{l} \label{system1}
y(t)+\lambda \int_0^1 k(t,s) y^{(m)}(s) y^{(n)}(s) ds = f(t),\\ \\
y(0)=a_0,\dots, y^{(l)}(0)=a_l,
\end{array}
\right.
\end{eqnarray}
and solutions  $y(t)$ of the Volterra integro-differential system given by
\begin{eqnarray}
\left \{
\begin{array}{l} \label{system2}
y(t)+\beta  \int_0^t g(t,s) y^{(m)}(s) y^{(n)}(s) ds=f(t),\\ \\
y(0)=a_0,\dots, y^{(l)}(0)=a_l.
\end{array}
\right.
\end{eqnarray}

Here $m, n $ are positive integers,  $l=\mathrm{max}(m,n)-1$,  $a_0,\dots,a_l$ are initial conditions, the parameters $\beta, \lambda$ and the functions $k(t, s), g(t, s)$ and $f(t)$ are known and belong to $L^2 [0, 1)$.  The function $y(t)$ as well as its derivatives $y^{(n)}$ and $y^{(m)}$ are unknown. We point out that System (\ref{system1}) is a particular case of Equation (\ref{genequation2}). \\

Hybrid functions have been applied extensively for solving differential  systems and proved to be a useful mathematical
tool. The pioneering work via hybrid functions was led by the authors in \cite{13H,18R}, who first derived an operational
matrix for the integral of the hybrid function vector, and paved the way for the hybrid function analysis of the dynamic systems. Since then, the hybrid functions' approach has been improved and used to approximate differential equations or
systems (see \cite{Hsiao, Malek-Basirat-Hashemi, Miza, Miza1, Miza2, Miza3, Miza4} and the references therein). \\

The novelty and the key point in solving Systems (\ref{system1}) and (\ref{system2}) are to use some useful properties of hybrid functions to derive a new approximation  $Y^{(n)}$ of the derivative  $y^{(n)} (t)$ of order $n$ of the solution $y(t)$ (see Lemma \ref{fund}). Hence,  Systems (\ref{system1}) and (\ref{system2}) can be converted into  reduced algebraic systems. \\

For arbitrary positive integers $q$ and $r$, the  set $\{b_{km}(t)\},\, k = 1, 2, \dots ,q, \; m = 0, 1, \dots,r- 1$ of hybrid functions will be used to approximate the solution $y_{r-1}(t)$ of the given system of integro-differential equations. This approximate solution  will involve Legnedre polynomials of degree $r-1$ defined on $q$ subintervals of $[0, 1]$.\\

This paper is organized as follows. In Section \ref{pre}, we introduce hybrid functions and its properties. In Section \ref{Main}, we describe the method for approximating solutions of the Fredholm and Volterra integro-differential Systems (\ref{system1}) and (\ref{system2}). An upper bound of the error is given in Section \ref{Error}, and finally numerical results are reported in Section \ref{Examples}.

\section{Preliminaries}\label{pre}

In this section, we define the Legendre polynomials $p_m(t)$, as well as block-pulse and  hybrid functions. We also recall functions' approximation in the Hilbert space $L^2[-1,1]$.  \\

The Legendre polynomials $p_m(t)$  are polynomials of degree $m$ defined on the interval $[-1,1]$ by
$$p_m(t)= \sum_{k=0}^{M} \frac{(-1)^k (2m-2k)!}{2^m k!(m-k)! (m-2k)!} t^{m-2k}, \;\; m \in \mathbb{N},$$
where $M=   \left \{
\begin{array}{l}
\frac{m}{2}, \qquad  \mbox{if} \; m \; \mbox{is even,} \\ \\
\frac{m-1}{2}, \quad \, \mbox{if} \; m \; \mbox{is odd.}\\
\end{array}
\right.
$\\

Equivalently, the Legendre polynomials are given by the recursive formula (see \cite{Basirat-Malek-Hashemi, Hashemi-Basirat, Malek-Basirat-Hashemi, Sho-Abadi-Golpar})
\[\begin{array}{l}
p_0(t)=1, \quad p_1(t)=t,\\\\
p_{m+1}(t)=\frac{2m+1}{m+1}tp_m(t)-\frac{m}{m+1}p_{m-1}(t), \;\; m=1,2,3,\dots.
\end{array}\]

The set $\{p_m(t); \;t = 0, 1, \dots\}$ is a complete orthogonal system in $L^2[-1, 1]$.

\begin{definition} \cite{Basirat-Malek-Hashemi, Hashemi-Basirat, Malek-Basirat-Hashemi}
For an arbitrary positive integer $q$, let $\{b_k(t)\}_{k=1}^q$ be the finite set of block-pulse functions on the interval $[0, 1)$ defined by
 \begin{eqnarray*}
b_k(t)=\left \{
\begin{array}{l}
1, \quad \mbox{if} \; \frac{k-1}{q} \le t < \frac{k}{q}, \\ \\
0, \quad \mbox{elsewhere.}\\
\end{array}
\right.
\end{eqnarray*}

\end{definition}

The block-pulse functions are disjoint and have the property of orthogonality on $[0, 1)$, since for $i,j=1,2,\dots, q$, we have:
 \begin{eqnarray*}
b_i(t) b_j(t)=\left \{
\begin{array}{l}
0, \qquad \;\mbox{if} \; i \neq j, \\ \\
b_i(t), \quad \mbox{if} \; i = j, \\
\end{array}
\right.
\end{eqnarray*}
and
 \begin{eqnarray*}
\< b_i(t),  b_j(t)\>=\left \{
\begin{array}{l}
0, \quad \mbox{if} \; i \neq j, \\ \\
\frac {1}{q}, \quad \mbox{if} \; i = j,\\
\end{array}
\right.
\end{eqnarray*}
where $\<., .\>$ is the scalar product given by $\<f,g\> = \int_0^1 f(t)g(t)dt$, for any functions $f, g \in L^2[0, 1)$.
\begin{definition}\cite{Basirat-Malek-Hashemi, Hashemi-Basirat, Malek-Basirat-Hashemi, Sho-Abadi-Golpar}
Let $r$ be an arbitrary positive integer. The set of hybrid functions $\{b_{km}(t)\},\, k = 1, 2, \dots ,q, \; m = 0, 1, \dots,r- 1$, where $k$ is the order for block-pulse functions, $m$ is the order for Legendre polynomials and $t$ is the normalized time, is defined on the interval $[0, 1)$ as
 \begin{eqnarray*}
b_{km}(t)=\left \{
\begin{array}{l}
p_m(2qt-2k+1), \qquad \mbox{if} \; \frac{k-1}{q} \le t < \frac{k}{q}, \\ \\
0, \qquad \qquad \qquad \qquad \quad \mbox{elsewhere}.\\
\end{array}
\right.
\end{eqnarray*}
\end{definition}

Since $b_{km}(t)$ is the combination of Legendre polynomials and block-pulse functions which are both complete and orthogonal, then the set of hybrid functions is a complete orthogonal system in $L^2[0, 1)$.\\

We are now able to define the vector function $B(t)$ of hybrid functions on $[0, 1)$  by  $$B(t) = \Big (B_1^T (t), \dots, B_q^T (t) \Big)^T,$$
where $B_i (t) = \Big (b_{i0} (t), \dots, b_{i(r-1)}(t)\Big )^T$, for $i=1,2, \dots, q$, and $V^T$ denotes the transpose of a vector $V$. \\

{\bf Function approximation \cite{Basirat-Malek-Hashemi, Hashemi-Basirat, Malek-Basirat-Hashemi, Sho-Abadi-Golpar}:}  Every function $f(t) \in L^2[0,1)$ can be approximated as
$$f(t) \simeq \sum_{k=1}^{q} \sum_{m=0}^{r-1}  f_{km}b_{km}(t),$$ where
$$f_{km}= \frac{\langle f(t), b_{km}(t)\rangle}{\langle b_{km}(t), b_{km}(t) \rangle}, \, \forall \, k=1,\dots,q, \; \forall \, m=0,\dots,r-1.$$
Thus,  
\begin{eqnarray}\label{fctF}
f(t) \simeq  F^T B(t) = B^T(t)F,
\end{eqnarray} 
 where $F$ is the $rq \times 1$ column vector having $f_{km}$ as entries.  In a similar way, any function $g(t,s) \in  L^2\big([0, 1) \times [0, 1)\big)$ can be approximated as
\begin{eqnarray}\label{matrixG}
g(t,s) \simeq B^T(t) G B(s),
\end{eqnarray} 
where $G = (g_{ij})$ is an $rq \times rq$ matrix given by
$$g_{ij}= \frac{\langle B_{(i)}(t), \langle g(t,s),B_{(j)}(s)\rangle \rangle}{\langle B_{(i)}(t),B_{(i)}(t)\rangle \langle B_{(j)}(s),B_{(j)}(s)\rangle}, \; i,j=1,2,\dots,rq,$$ and $B_{(i)}(t)$  \big(resp. $B_{(j)}(s)$\big) denotes the $i^{th}$ component \big(resp. the $j^{th}$ component$\big)$ of $B(t)$ \big(resp. $B(s)$\big). \\

{\bf Operational matrix of integration \cite{Basirat-Malek-Hashemi, Hashemi-Basirat, Malek-Basirat-Hashemi, Sho-Abadi-Golpar}:} The integration of the vector function $B(t)$ may be approximated by $\int_0^t B(t') dt' \simeq P B(t)$, where $P$
is an $rq \times rq$ matrix known as the operational matrix of integration and given by

 \[ P =  \left( \begin{array}{ccccc}
E & H & H & ... & H \\
0 & E & H & ... & H \\
0 & 0 & E & ... & ... \\
. & . & . & ...& H \\
0 & 0 & 0 & ... & E
\end{array} \right)\]

where $H$  and $E$ are $r \times r$ matrices defined by

\[  H = \frac 1q  \left( \begin{array}{ccccc}
1 & 0 & 0 & ... & 0 \\
0 & 0 & 0 & ... & 0 \\
0 & 0 & 0 & ... & 0 \\
. & . & . & ...& 0 \\
0 & 0 & 0 & ... & 0
\end{array} \right)\]
 and

\[ E= \frac{1}{2q}\left( \begin{array}{cccccccc}
1 & 1 & 0 & 0 & ... & 0& 0 & 0 \\

-\frac 13 & 0 & \frac 13 & 0 & ... & 0& 0 & 0 \\
0 & -\frac 15 & 0 & \frac 15 & ... & 0& 0 & 0 \\
.. & .. & .. & .. & ... & ..& .. & .. \\
0 & 0 & 0 & 0 & ... & -\frac{1}{2r-3}& 0 & \frac{1}{2r-3} \\
0 & 0 & 0 & 0 & ... & 0& -\frac{1}{2r-1} & 0
\end{array} \right).\] \\\\
{\bf The integration of two hybrid functions \cite{Basirat-Malek-Hashemi, Hashemi-Basirat, Malek-Basirat-Hashemi, Sho-Abadi-Golpar}:} The integration of the cross product of two hybrid function vectors is given by $L = \int_0^1 B(t) B^T(t) dt,$ where $L$ is the $rq \times rq$ diagonal matrix defined by
 \[ L= \left( \begin{array}{cccccccc}
D & 0 & 0 & 0 & ... & 0& 0 & 0 \\
0 & D  & 0 & 0 & ... & 0 & 0 & 0 \\
.. & .. & .. & .. & ... & .. & .. & .. \\
.. & .. & .. & .. & ... & D & .. & .. \\
0 & 0 & 0 & 0 & ... & 0& D & 0 \\
0 & 0 & 0 & 0 & ... & 0& 0 & D
\end{array} \right)\]

where $D$ is the $r \times r$ matrix given by

 \[ D= \frac 1q \left( \begin{array}{cccccccc}
1 & 0 & 0 & 0 & ... & 0& 0 & 0 \\
0 & \frac 13  & 0 & 0 & ... & 0 & 0 & 0 \\
.. & .. & .. & .. & ... & .. & .. & .. \\
0 & 0 & 0 & 0 & ... & 0& 0 & \frac{1}{2r-1}
\end{array} \right).\] \\

{\bf  The matrix $\widetilde{C}$ associated to a vector $C$ \cite{Basirat-Malek-Hashemi, Hsiao, Malek-Basirat-Hashemi, Sho-Abadi-Golpar}:}  For any $rq \times  1$ vector $C$, we define the $rq \times rq$ matrix $\widetilde{C}$ such that
$$B(t)B^T(t) C= \widetilde{C} B(t).$$
$\widetilde{C}$ is called the coefficient matrix. In \cite{Hsiao}, Hsiao computed the matrix $\widetilde{C}$ for $r=2$ and $q=8$, while the authors in \cite{Basirat-Malek-Hashemi} considered the case of $r=4$ and $q=3$.\\

{\bf The vector $\widehat S$ associated to a matrix $S$}: For any $rq \times rq$ matrix $S$, we define the  $1 \times rq$ row vector $\widehat { S }$ such that $B^T(t) S B(t)=\widehat S B(t).$ For instance, let $S$ be a $12\times 12$ matrix with coefficients $s_{11},s_{12},\dots, s_{(12)(11)}, s_{(12)(12)}$. After developing and comparing the two sides of the equation $B^T(t) S B(t)=\widehat S B(t)$, we deduce that the $1\times 12$ row vector $\widehat{S}$ is given by:

$$\widehat{S}=\begin{pmatrix}
s_{11}+\frac{1}{3}s_{22}+\frac{1}{5}s_{33} \\
s_{12}+s_{21}+\frac{2}{5}s_{23}+\frac{2}{5}s_{32} \\
s_{13}+s_{31}+\frac{2}{3}s_{22}+\frac{2}{7}s_{33}\\
s_{44}+\frac{1}{3}s_{55}+\frac{1}{5}s_{66} \\
s_{45}+s_{54}+\frac{2}{5}s_{56}+\frac{2}{5}s_{65} \\
s_{46}+s_{64}+\frac{2}{3}s_{55}+\frac{2}{7}s_{66}\\
s_{77}+\frac{1}{3}s_{88}+\frac{1}{5}s_{99} \\
s_{78}+s_{87}+\frac{2}{5}s_{89}+\frac{2}{5}s_{98} \\
s_{79}+s_{97}+\frac{2}{3}s_{88}+\frac{2}{7}s_{99}\\
s_{(10)(10)}+\frac{1}{3}s_{(11)(11)}+\frac{1}{5}s_{(12)(12)} \\
s_{(10)(11)}+s_{(11)(10)}+\frac{2}{5}s_{(11)(12)}+\frac{2}{5}s_{(12)(11)} \\
s_{(10)(12)}+s_{(12)(10)}+\frac{2}{3}s_{(11)(11)}+\frac{2}{7}s_{(12)(12)}\\

\end{pmatrix}.
$$


\section{Main results}\label{Main}
In this section, we approximate solutions $y(t)$ of Systems (\ref{system1}) and (\ref{system2}). For this, we need the approximation of $y^{(n)} (t)$.
\begin{lemma}\label{fund}
Let $y(t)$ be a function and consider its approximation
$y_{r-1}(t)=    Y^{T} B(t) = B^{T}(t) Y$. If $Y^{(n)}$ denotes the approximation of $y^{(n)}(t)$, then for any $n \geq 1$, we have:
$$Y^{(n)} = J^n Y - \sum_{k=1}^n  J^k Y^{(n-k)}_0,$$
where $J =(P^{T})^{-1}$ and $Y_0^{(i)}$ are the approximations of the initial conditions $y_0^{(i)}$, for
$i= 0,\dots, n~-~1$.
\end{lemma}

\begin{proof}
By the Fundamental Theorem of Calculus, we have $$y(t) = \int_{0}^t y' (s) ds + y(0).$$
Approximating $y(t), y'(t)$ and $y_0(t)$, we get
\begin{equation*}
\begin{split}
Y^T B(t) & \simeq \int_0^t (Y^{(1)})^T B(s)ds + Y_0^T B(t)\\
 & \simeq (Y^{(1)})^T \int _0^t B(s) ds + Y_0^T B(t) \\
 & \simeq (Y^{(1)})^T P B(t)  + Y_0^T B(t) \\
& \simeq \big((Y^{(1})^TP +Y_0^T \big)B(t).
\end{split}
\end{equation*}

Thus,
$Y^T=(Y^{(1)})^TP+Y_0^T$ and so $Y=P^TY^{(1)} +Y_0$, giving that $Y^{(1)}= J(Y-Y_0)$ and the result is true for $n=1$.\\

By induction, assume that the result is true for $n$ and prove it for $n+1$. We have
\begin{equation*}
\begin{split}
Y^{(n+1)} & =J \big(Y^{(n)}-Y^{(n)}_0 \big)=  J \big(J^n Y - \sum_{k=1}^n  J^k Y^{(n-k)}_0- Y^{(n)}_0 \big) \\
& = J^{n+1}Y- \sum_{k=1}^n J^{k+1}Y_0^{(n-k)}-JY_0^{(n)} = J^{n+1}Y- \sum_{k=0}^n J^{k+1}Y_0^{(n-k)} \\
& = J^{n+1}Y- \sum_{k=1}^{n+1} J^k Y_0^{(n+1-k)},
\end{split}
\end{equation*}
which is the desired result.
\end{proof}

We are now ready to approximate solutions of Systems (\ref{system1}) and  (\ref{system2}).

\subsection{Approximated Solution of the Fredholm Integro-Differential System (\ref{system1})}
Using the approximations (\ref{fctF}) and (\ref{matrixG}) of functions of one and two variables, System (\ref{system1}) can be approximated as
\begin{equation*}
\begin{split}
& B^T(t) Y + \lambda \int_0^ 1 B^T(t) K B(s) B^T(s) Y^{(m)} B^T(s) Y^{(n)} ds = B^T(t) F \\
 & \Longrightarrow  Y  + \lambda K \int_0^ 1 B(s) B^T(s) Y^{(m)} B^T(s) Y^{(n)} ds =  F \\
 & \Longrightarrow  Y  + \lambda K \int_0^ 1   \oversortoftilde{Y^{(m)}}B(s)B^T(s) Y^{(n)} ds =  F \\
 & \Longrightarrow  Y  + \lambda K  \oversortoftilde{ {Y^{(m)}}}\Big (\int_0^ 1 B(s) B^T(s) ds \Big ) Y^{(n)}= F \\
& \Longrightarrow Y+ \lambda K \oversortoftilde{ {Y^{(m)}}} L Y^{(n)} =F.
\end{split}
\end{equation*}
Using Lemma \ref{fund}, the last equation becomes
\begin{equation}
Y+ \lambda K \Big [  \oversortoftilde{ J^m Y -\sum_{k=1}^m J^k Y_0^{(m-k)}} \Big] L \Big[ {J^n Y -\sum_{k=1}^n J^k Y_0^{(n-k)}}\Big]=F. \label{FredSol}
\end{equation}

This is a nonlinear system of $rq$ equations in $rq$ variables which can be solved by any iterative method.

\subsection{Approximated Solution of the Volterra Integro-Differential System (\ref{system2})}
Using the approximations (\ref{fctF}) and (\ref{matrixG}), System (\ref{system2}) can be approximated as:
\begin{equation*}
\begin{split}
& Y^T B(t) + \beta \int_0^t B^T(t) G B(s)  B^T(s) Y^{(m)} B^T(s)   Y^{(n)} ds = F^T B(t) \\
&  \Longrightarrow  Y^T B(t)  + \beta  B^T(t) G \int_0^t   B(s)  B^T(s) Y^{(m)} B^T(s)   Y^{(n)} ds = F^T B(t) \\
&    \Longrightarrow Y^T B(t)  + \beta B^T(t) G \oversortoftilde{ Y^{(m)}} \int_0^t B(s)  B^T(s) Y^{(n)}  ds = F^T B(t) \\
& \Longrightarrow  Y^T B(t)  + \beta B^T(t) G \oversortoftilde{ Y^{(m)}} \; \oversortoftilde{ Y^{(n)}} \int_0^t B(s) ds=  F^T B(t) \\
& \Longrightarrow Y^T B(t) + \beta B^T(t) G \oversortoftilde{ Y^{(m)}} \;\oversortoftilde{Y^{(n)}} P B(t)= F^T B(t).
\end{split}
\end{equation*}

Consider the matrix $S := G \oversortoftilde{Y^{(m)}} \; \oversortoftilde{ Y^{(n)}} P $. We obtain
\begin{equation*}
\begin{split}
& Y^T B(t) + \beta B^T(t) S B(t)= F^T B(t)\\
& \Longrightarrow Y^T B(t) +\beta \widehat{S}B(t)=F^T B(t).
\end{split}
\end{equation*}
Hence,
\begin{equation}
Y^T  +\beta \widehat{S}=F^T. \label{VoltSol}
\end{equation}

Finally, using Lemma~\ref{fund} for $Y^{(m)}$ and $Y^{(n)}$, we get a nonlinear system  which can be solved by any iterative method.

\section{Error Analysis}\label{Error}
We assume that the function $y(t)$ is sufficiently smooth on the interval $[0, 1]$. Suppose that $t_0, t_1, \cdots, t_\mu$ are the roots of $\mu+1$ degree shifted Chebyshev polynomial $P_\mu(t)$ in $[0, 1]$ that interpolates $y(t)$ at the nodes $t_i$, $0\leq i \leq \mu$. 
The error in the interpolation is given in \cite{er1} by
\begin{equation}
y(t)-P_\mu(t) = \frac{d^{\mu+1}}{dt^{\mu+1}} (y(\delta)) \frac{\Pi_{i=0}^\mu (t-t_i)}{(\mu+1)!}, \label{apr1}
\end{equation}
for some $\delta\in [0, 1]$. This shows that 
\begin{equation}
\vert y(t)-P_\mu(t) \vert \leq \frac{M}{2^{2\mu+1}(\mu+1)!},
\end{equation}
where $M=\max_{t\in [0,1]}\big\vert \frac{d^{\mu+1}}{dt^{\mu+1}}(y(t)) \big\vert$.\\

We recall here that the $L_2$ norm of a function $y:[0, 1] \longrightarrow \mathbb{R}$ is given by $\Vert y \Vert_2 = \Big(\int_0^1 y^2(t) dt\Big)^{\frac 12}$.

\begin{thm}
	If $y_\mu(t) = B^T(t) Y$ is the best approximation of the solution $y(t)$ obtained using Legendre polynomials then 
	\begin{equation}
	\Vert y(t) - y_\mu(t) \Vert_2 \leq \frac{M}{2^{2\mu+1}(\mu+1)!}, \label{est}
	\end{equation}
	for some constant $M>0$.
\end{thm}
\begin{proof} Let $X_\mu$ be the space of all polynomials of degree less than or equal to $\mu$. Since $y_\mu$ is the best approximation to $y$, $\Vert y - y_\mu \Vert_2 \leq \Vert y - g \Vert_2$, for any arbitrary polynomial $g$ in $X_\mu$. Therefore by using (\ref{apr1}), we get
\begin{eqnarray}
\Vert y - y_\mu \Vert_2^2 &=& \int_0^1 \left(y(t)-y_\mu(t)\right)^2 dt\nonumber \\
&\leq& \int_0^1 \left(y(t)-P_\mu(t)\right)^2 dt\nonumber \\
&\leq& \left(\frac{M}{2^{2\mu+1}(\mu+1)!}\right)^2 \int_0^1  dt = \nonumber \left(\frac{M}{2^{2\mu+1}(\mu+1)!}\right)^2.
\end{eqnarray} 
The result is obtained by taking square-root on both sides.
\vspace{-0.05cm}
\end{proof}

\section{Numerical examples}\label{Examples}

In this section, we apply the methods described in Section \ref{Main} to some numerical examples to solve Systems (\ref{system1}) and (\ref{system2}).

\begin{example}
Consider the following Fredholm integro-differential system
\begin{eqnarray}\label{ex1}
\left \{
\begin{array}{l} 
y(t)+ \int_0^1 e^{t-s} y(s) y^{'}(s) ds =e^{t+1},\\ \\
y(0)=1.
\end{array}
\right.
\end{eqnarray}
\end{example}
Comparing with the standard form of System (\ref{system1}), we get $\lambda=1,$ $k(t, s)=e^{t-s}$, $m=0,$ $n=1$,  $f(t)=e^{t+1}$, $l=0$ and $a_0=1$.\\

\textbf{Case 1:}  First we consider $r=2$ and $q =1$.
It can be verified that $B(t) = (1, 2t-1)^T$ and
\begin{equation*}
K=\left( \begin{array}{ccc}
e+ \frac 1e - 2 &  3 (e+ \frac 3e -4) \\
 3 (-e + 4 - \frac 3e)& 9 (6-e -\frac 9e)
\end{array} \right).
\end{equation*}

  The matrix approximations $F$ of the function $f(t)=e^{t+1}$ and $Y_0$  of $y(0)$ are respectively given by
  \[  F =   \left( \begin{array}{cc}
  e^2- e \\
  -3 e^2+ 9 e
  \end{array} \right)\]
 and
 \[  Y_0 =   \left( \begin{array}{cc}
 1\\
 0
 \end{array} \right).\]
From Equation (\ref{FredSol}), we deduce
 \[  Y =   \left( \begin{array}{cc}
 e-1\\
 -3e+ 9
 \end{array} \right).\] \\
 Using the approximation $y_1(t) = Y^{T} B(t) = B^{T}(t) Y$, we get $y_1(t)=4e-10+(18-6e)t$. In Figure \ref{Fig.1}, we compare this approximate solution $y_1(t)$ with the exact solution $e^t$. The absolute errors at various values of $t$ are shown in Table \ref{table1}.\\

\vspace{-0.1cm}

 \begin{figure}[h]
 \centering
   \includegraphics[scale=1]{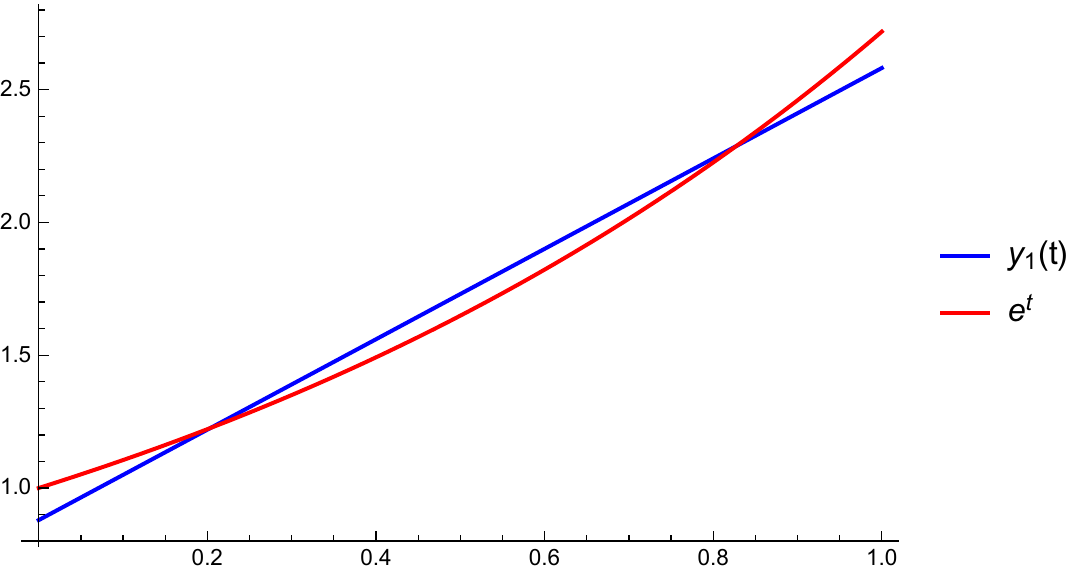}
   \caption{Comparison of approximate and exact solutions of System (\ref{ex1}) for the case $r=2$ and $q =1$.}
 \label{Fig.1}
 \end{figure}

 \begin{table}[h]
 \centering
 \begin{tabular}{|c|c|c|c|}
 \hline
$t$ & Error & $t$ & Error\\
\hline
0.0 &  0.120825 & 0.5 & 0.08146\\
\hline
0.1 & 0.05579 & 0.6 & 0.07827\\
\hline
0.2 & 0.00182 & 0.7 & 0.05683\\
\hline
0.3 & 0.03992 & 0.8 & 0.01525\\
\hline
0.4 & 0.06816 & 0.9 & 0.04861\\
\hline
 \end{tabular}
 \caption{Absolute errors in solution of System (\ref{ex1}) with $r=2$ and $q =1$.}
 \label{table1}
 \end{table}
This is degree 1 approximation, therefore $\mu=1$. Further, $M=\max_{[0, 1]}y'(t) = e $. The error estimate by using (\ref{est}) is $M/2^4 \approxeq 0.169893$. It can be checked from Table \ref{table1} that our computed~values are less than this error-bound.\\

 \textbf{Case 2:}  Now, we consider $r=3$ and $q =4$. A $12\times 12$ matrix $K$ is given by:
\begin{eqnarray}
K&=&\left(\begin{matrix} 
1.0052 & -0.1255 & 0.0052 & 0.7829 & -0.0978 & 0.0041 \nonumber\\
0.1255 & -0.0157 & 0.0007 & 0.0978 & -0.0122 & 0.0005 \nonumber\\ 
0.0052 & -0.0007 & 0.0 & 0.00401 & -0.0005 & 0.0 \nonumber\\
1.2907 & -0.1612 & 0.0067 & 1.0052 & -0.1255 & 0.0052 \nonumber\\
0.1612 & -0.0202 & 0.0008 & 0.1255 & -0.0157 & 0.0007 \nonumber\\
0.0067 & -0.0008 & 0.0    & 0.0052 & -0.0007 & 0.0  \nonumber\\
1.6573 & -0.2070 & 0.0086 & 1.2907 & -0.1612 & 0.0067 \nonumber\\
0.2070 & -0.0258 & 0.0011 & 0.1612 & -0.0201 & 0.0008 \nonumber\\
0.0086 & -0.0011 & 0.0 & 0.0067 & -0.0008 &  0.0 \nonumber\\
2.1281 & -0.2658 & 0.0111 & 1.6573 & -0.2070 & 0.0086 \nonumber\\
0.2657 & -0.0332 & 0.0014 & 0.2070 & -0.0258 & 0.0011 \nonumber\\
0.0111 & -0.0014 & 0.0 & 0.0086 & -0.0011 & 0.0\nonumber
\end{matrix}
\right. \\
\nonumber\\
&& \left. \begin{matrix}
0.6097 & -0.0761 & 0.0032 & 0.4748 & -0.0593 & 0.0025 \nonumber\\
 0.0761 & -0.0095 & 0.0003 & 0.0593 & -0.0074 & 0.0003 \nonumber\\
 0.0032 & -0.0004 & 0.0    & 0.0025 & -0.0003 & 0.0 \nonumber\\
 0.7829 & -0.0978 & 0.0040 & 0.6097 & -0.0761 & 0.0032 \nonumber\\
 0.0978 & -0.0122 & 0.0005 & 0.0761 & -0.0095 & 0.0004 \nonumber\\
 0.0041 & -0.0005 & 0.0    & 0.0032 & -0.0004 & 0.0    \nonumber\\
 1.0052 & -0.1255 & 0.0052 & 0.7829 & -0.0978 & 0.00401 \nonumber\\
 0.1255 & -0.0157 & 0.0007 & 0.0978 & -0.0122 & 0.0005 \nonumber\\
 0.0052 & -0.0007 & 0.0 & 0.0041 & -0.0005 & 0.0 \nonumber\\
 1.2907 & -0.1612 & 0.0067 & 1.0052 & -0.1255 & 0.0052 \nonumber\\
 0.1612 & -0.0201 & 0.0008 & 0.1255 & -0.0157 & 0.0007 \nonumber\\
 0.0067 & -0.0008 & 0.0 & 0.0052 & -0.0007 & 0.0 \nonumber
\end{matrix}\right)
\end{eqnarray}
 Other approximations in this case are as below:
 
 \begin{eqnarray*}
   B(t)  &=& \left( \chi_{[0, 1/4)},  (-1+8t) \chi_{[0, 1/4)}, (1-24t+96t^2) \chi_{[0, 1/4)}, \chi_{[1/4, 1/2)}, \right.\label{vecB}\\
   &&   (-3+8t) \chi_{[1/4, 1/2)}, (13-72t+96t^2) \chi_{[1/4, 1/2)}, \chi_{[1/2, 3/4)}, (-5+8t)\chi_{[1/2, 3/4)},\nonumber\\
  && \left. (37-120t+96t^2) \chi_{[1/2, 3/4)}, \chi_{[3/4, 1)}, (-7+8t) \chi_{[3/4, 1)}, (73-168t+96t^2)\chi_{[3/4, 1)}\right), \nonumber
   \end{eqnarray*}
 (where $\chi_A$ is the characteristic function of a set $A$)
 
 \begin{eqnarray}
 F &=& \left(3.08824, 0.385629, 0.0160607, 3.96538, 0.495157, 0.0206224, \right.\nonumber\\
 &&\, \, \left.5.09165,   0.635795, 0.0264796, 6.53781, 0.816377, 0.0340005\right)^T, \nonumber
 \end{eqnarray}

\begin{equation*}
 Y_0 = \left( 1, 0, 0, 1, 0, 0, 1, 0, 0, 1, 0, 0\right)^T, 
 \end{equation*}
and
\begin{eqnarray}
 Y &=& \left(  1.1361, 0.141865, 0.00590841, 1.45878, 0.182158, 0.00758655, \right.\nonumber\\
 &&\, \, \left. 1.87312, 0.233896, 0.00974132, 2.40513, 0.300328, 0.0125081\right)^T. \nonumber
 \end{eqnarray}
  The approximate solution of (\ref{ex1}) is given by
  \begin{eqnarray}
  y_2(t) &=& 1.1361 \chi_{[0, 1/4)}+1.45878 \chi_{[1/4, 1/2)} + 1.87312 \chi_{[1/2, 3/4)}+2.40513 \chi_{[3/4, 1)} \nonumber\\
  &&+0.3 (-7+8t) \chi_{[3/4, 1)}+ 0.233896 (-5+8t)\chi_{[1/2, 3/4)} \nonumber\\
  &&  +0.182158 (-3+8t) \chi_{[1/4, 1/2)} + 0.14187(-1+8t) \chi_{[0, 1/4)} \nonumber\\
    && + 0.0125(73-168t+96t^2)\chi_{[3/4, 1)}+ 0.0097(37-120t+96t^2) \chi_{[1/2, 3/4)} \nonumber\\
      &&+ 0.00759 (13-72t+96t^2) \chi_{[1/4, 1/2)} + 0.00591 (1-24t+96t^2) \chi_{[0, 1/4)}.\nonumber
  \end{eqnarray}
   Figure \ref{Fig.2} shows the graphs of the approximate solution $y_2(t)$ and the exact solution $e^t$. The absolute errors at various values of $t$ are given in Table \ref{table2}. It can be observed that, in this case, the approximate solution is well in agreement with the exact solution. Further, using (\ref{est}), the error bound is $0.01416$. Table \ref{table2} shows that our computed maximum value is $0.000145961$.

 \begin{figure}[h]
  \centering
    \includegraphics[scale=1]{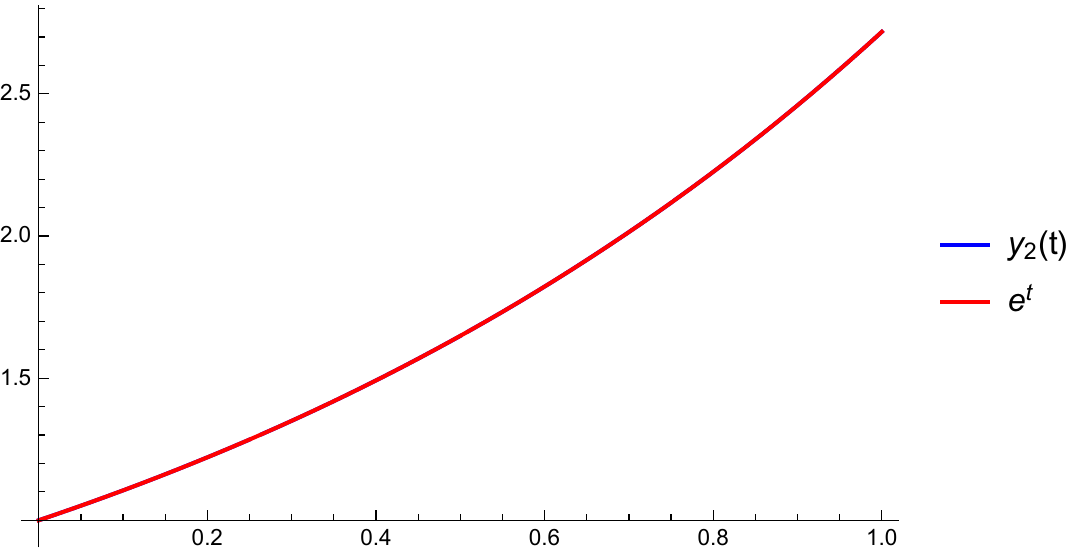}
    \caption{Comparison of approximate and exact solutions of System (\ref{ex1}) for the case $r=3$ and $q =4$.}
    \label{Fig.2}
  \end{figure}
 
  \begin{table}[h]
  \centering
  \begin{tabular}{|c|c|c|c|}
  \hline
 $t$ & Error & $t$ & Error\\
 \hline
   0.0 & 0.000145961 &  0.5 &  0.000240649\\
  \hline
 0.1 & 0.0000409679 & 0.6 & 0.0000675446\\
 \hline
 0.2 & 0.0000553281 & 0.7 & 0.0000912206\\
 \hline
 0.3 & 0.0000656897 & 0.8 & 0.000108304\\
 \hline
 0.4 & 0.0000536172 & 0.9 & 0.0000883998\\
 \hline
  \end{tabular}
  \caption{Absolute errors in solution of System (\ref{ex1}) with $r=3$ and $q =4$.}
  \label{table2}
  \end{table} 

\begin{example}
Consider the following Volterra integro-differential system
\begin{eqnarray}\label{ex2}
\left \{
\begin{array}{l} 
y(t)+ \int_0^t \sin(t-s) y(s) y^{'}(s) ds = 2 t^3 + t^2 -12 t + 12 \sin(t),\\ \\
y(0)=0.
\end{array}
\right.
\end{eqnarray}
\end{example} 

Comparing with the standard form of System (\ref{system2}), we get $\beta=1,$ $g(t, s)=\sin(t-s)$, $m=0,$ $n=1$,  $f(t)=2 t^3 + t^2 -12 t + 12 \sin(t)$, $l=0$ and $a_0=0$. We take $r=3$ and $q =4$. Following the procedure described in Section 2, we get 
  \begin{eqnarray}
   F &=& \left(0.0208496, 0.0312848, 0.0104457, 0.146854, 0.0951443, 0.0109877, \right.\nonumber\\
   &&\, \, \left. 0.406543, 0.166108, 0.0129514, 0.824576, 0.255306, 0.0171862\right)^T, \nonumber
   \end{eqnarray}
          
 \begin{eqnarray}
 G &=&\left(\begin{matrix} 
   0 & -0.1245 & 0 & -0.2461 & -0.1206 & 0.0013 \nonumber\\
   0.1245 & 0 & -0.0006 & 0.1206 & -0.0039 & -0.0006\nonumber\\ 
  0 & 0.0006 & 0 & 0.0013 & 0.0006 & 0 \nonumber\\
  0.2461 & -0.1206 & -0.0013 & 0 & -0.1245 & 0  \nonumber\\
   0.1206 & 0.0039 & -0.0006 & 0.1245 & 0 & -0.0006 \nonumber\\ 
   -0.0013 & 0.0006 & 0 & 0 & 0.0006 & 0 \nonumber\\
  0.4769 & -0.1092 & -0.0025 & 0.2461 & -0.1206 & -0.0013  \nonumber\\
   0.1092 & 0.0075 & -0.0006 & 0.1206 & 0.0039 & -0.0006 \nonumber\\ 
  -0.0025 & 0.0006 & 0 & -0.0013 & 0.0006 & 0 \nonumber\\
  0.6781 & -0.0911 & -0.0035 & 0.4769 & -0.1092 & -0.0025  \nonumber\\
   0.0911 & 0.0106 & -0.0005 & 0.1092 & 0.0075 & -0.0006 \nonumber\\ 
   -0.0035 & 0.0005 & 0 & -0.0025 & 0.0006 & 0 \nonumber
 \end{matrix}
 \right. \\
 \nonumber\\
 && \left. \begin{matrix}
  -0.4769 & -0.1092 & 0.0025 & -0.6781 & -0.0911 & 0.0035 \nonumber\\
  0.1092 & -0.0075 & -0.0006 & 0.0912 & -0.0106 & -0.0005 \nonumber\\ 
  0.0025 & 0.0006 & 0 & 0.0035 & 0.0005 & 0 \nonumber\\
 -0.2461 & -0.1206 & 0.0013 & -0.4769 & -0.1092 & 0.0025  \nonumber\\
  0.1206 & -0.0039 & -0.0006 & 0.1092 & -0.0075 & -0.0006 \nonumber\\ 
   0.0013 & 0.0006 & 0 & 0.0025 & 0.0006 & 0 \nonumber\\
  0 & -0.1245 & 0 & -0.2461 & -0.1206 & 0.0013  \nonumber\\
    0.1245 & 0 & -0.0006 & 0.1206 & -0.0039 & -0.0006 \nonumber\\ 
    0 & 0.0006 & 0 & 0.0013 & 0.0006 & 0 \nonumber\\
   0.2461 & -0.1206 & -0.00128 & 0 & -0.1245 & 0  \nonumber\\
   0.1206 & 0.0039 & -0.0006 & 0.1245 & 0 & -0.0006   \nonumber\\ 
     -0.0013 & 0.0006 & 0 & 0 & 0.0006 & 0  \nonumber
 \end{matrix}\right),
 \end{eqnarray}  
   
 \begin{equation*}
  Y_0 = \left( 0, 0, 0, 0, 0, 0, 0, 0, 0, 0, 0, 0\right)^T, 
  \end{equation*}  
and $B(t)$ is same as given in the previous example. Using Equation (\ref{VoltSol}), we get
\begin{eqnarray}
 Y &=& \left(  0.0208333, 0.0312487, 0.0104375, 0.145833, 0.0937492, 0.0104781, \right.\nonumber\\
 &&\, \, \left. 0.395833, 0.15625, 0.0105145, 0.770834, 0.218753, 0.010543\right)^T. \nonumber
 \end{eqnarray}

The approximate solution of System (\ref{ex2}) is given by
 \begin{eqnarray*} 
  y_2(t) &=& 0.0208333 \chi_{[0, 1/4)} + 0.145833 \chi_{[1/4, 1/2)} \\ 
    && + 0.395833 \chi_{[1/2, 3/4)} + 0.770834 \chi_{[3/4, 1)} \nonumber\\
      && +0.218753 (-7+8t) \chi_{[3/4, 1)}+ 0.15625 (-5+8t)\chi_{[1/2, 3/4)}\nonumber\\
  &&  +0.0937492 (-3+8t) \chi_{[1/4, 1/2)} + 0.0312487(-1+8t) \chi_{[0, 1/4)} \nonumber\\
    && + 0.010543(73-168t+96t^2)\chi_{[3/4, 1)}+ 0.0105145(37-120t+96t^2) \chi_{[1/2, 3/4)} \nonumber\\
      &&+ 0.0104781 (13-72t+96t^2) \chi_{[1/4, 1/2)} + 0.0104375 (1-24t+96t^2) \chi_{[0, 1/4)}.\nonumber
  \end{eqnarray*}
The approximate solution $y_2(t)$ is compared with the exact solution $t^2$ in Figure \ref{Fig3}. Table \ref{table3} shows the absolute error in the solution at different values of $t$.

 \begin{figure}[h]
  \centering
    \includegraphics[scale=1]{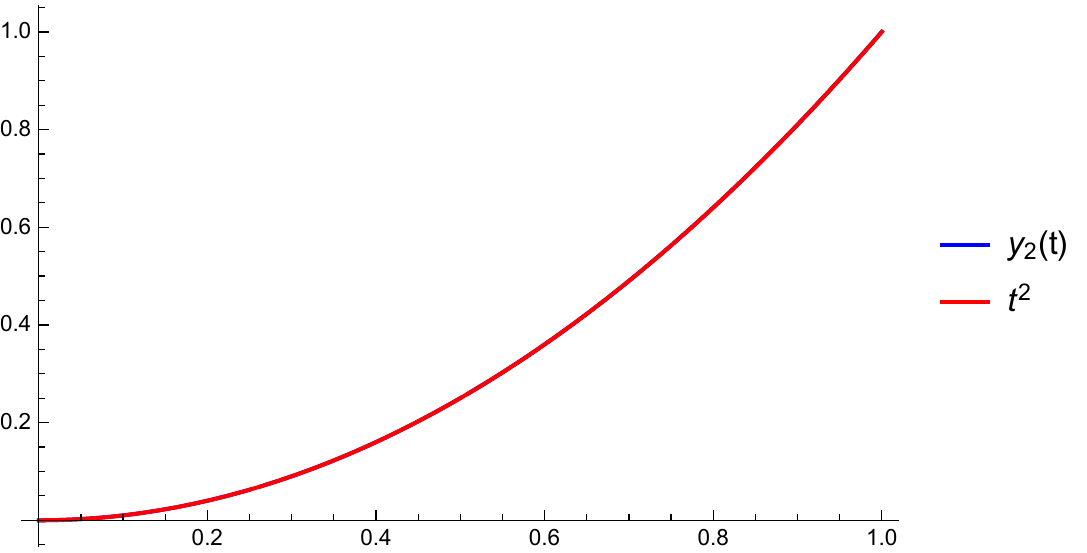}
    \caption{Comparison of approximate and exact solutions of System (\ref{ex2}) for the case $r=3$ and $q =4$.}
    \label{Fig3}
  \end{figure}
 
  \begin{table}[h]
  \centering
  \begin{tabular}{|c|c|c|c|}
  \hline
 $t$ & Error & $t$ & Error\\
 \hline
   0.0 & 0.0000221689 &  0.5 &  0.0000975226\\
  \hline
 0.1 & 8.9$\times 10^{-6}$ & 0.6 & 0.0000429716\\
 \hline
 0.2 & 4.99$\times 10^{-8}$  & 0.7 & 4.32$\times 10^{-6}$\\
 \hline
 0.3 & 3$\times 10^{-6}$ & 0.8 & 3.76$\times 10^{-6}$\\
 \hline
 0.4 & 0.0000271471 & 0.9 & 0.0000547743\\
 \hline
  \end{tabular}
  \caption{Absolute errors in solution of System (\ref{ex2}) with $r=3$ and $q =4$.}
  \label{table3}
  \end{table} 
Using (\ref{est}), the error-bound for this example is approximately $0.0104167$ and our value from Table \ref{table3} is $0.0000271471$.

\section{Conclusion}
 In this work, we have discussed an efficient method to solve a class of Fredholm and Volterra integro-differential equations. Our method is based on a new approximation for the derivatives of the equations' solutions using  hybrid and block-pulse functions. The absolute errors reported in tables show that the approximate solution is in a good agreement with the exact solution.  It is verified in each example that the practical error in our method is less than the theoretical error-bound. In fact, this method is highly efficient, very easy and a powerful mathematical tool for finding the numerical solution of  some class of Fredholm and Volterra integro-differential equations. The approximate solutions are found by using the computer code written in Matlab. The method is computationally attractive, and applications are demonstrated through illustrative examples. For future research works, we can use this method to solve various kinds of problems such as higher dimensional problems, stochastic integro-differential equations and partial integro-differential equations of fractional order with additional work.


\begin{thebibliography}{99}


\bibitem{VIM2} S. Abbasbandy and E. Shivanian, \emph{Application of variational iteration method for nth-order integro-differential equations}, Zeitschrift fur Naturforschung A, 64 (6-7) (2009), 439-444.

\bibitem{eng1} M.A. Abdou, \emph{On asymptotic methods for Fredholm-Volterra integral equation of the second kind in contact problems}, J. Comput. Appl. Math., 154 (2003), 431-446.

\bibitem{Basirat-Malek-Hashemi} B. Basirat, K. Maleknejad and E. Hashemizadeh, \emph{Operational matrix approach for the nonlinear Volterra-Fredholm integral equations: Arising in physics and engineering}, International Journal of Physical Sciences, 7.2 (2012), 226-233 .

\bibitem{phs}  F. Bloom, \emph{Asymptotic bounds for solutions to a system of damped integro-differential equations of electromagnetic theory}, J. Math. Anal. Appl., 73 (1980), 524-542.



  \bibitem{NIM} V. Daftardar-Gejji and H. Jafari, \emph{An iterative method for solving non-linear functional equations}, J. Math. Anal. Appl., 316 (2006), 753-763.
\bibitem{neuro} L.M. Delves and J.L. Mohamed, \emph{Computational methods for integral equations}, Cambridge University Press, Cambridge (1985).

\bibitem{eng2} L.K. Forbes, S. Crozier and D.M. Doddrell, \emph{Calculating current densities and fields produced by shielded magnetic resonance imaging probes}, SIAM J. Appl. Math., 57 (2) (1997), 401-425.

\bibitem{mol} J. Frankel, \emph{A Galerkin solution to regularized Cauchy singular integro-differential equation}, Quart. Appl. Math., 52 (2) (1995), 145-258.


\bibitem{HeHom} A. Golbabai and M. Javidi, \emph{Application of He's  homotopy perturbation method for $n^{th}$-order integro-differential equations}, Appl. Math. Comput., 190 (2007), 1409-1416.
\bibitem{ar} M.A. Golberg, \emph{The convergence of a collocations method for a class of Cauchy singular integral equations}, J. Math. Appl., 100 (1984), 500-512.

\bibitem{Hashemi-Basirat} E. Hashemizadeh and B. Basirat, \emph{An efficient computational method for the system of linear Volterra integral equations by means of hybrid functions}, Mathematical Sciences, 5 (4) (2011), 355-368.

\bibitem{hemada} A. A. Hemeda, \emph{New iterative method: application to $n^{th}$-order integro-differential equations},  International Mathematical Forum, Vol. 7 (47) (2012), 2317-2332.

\bibitem{Hybgen} J. Hou and C. Yang, \emph{Numerical method in solving Fredholm
integro-differential equations by using hybrid function operational matrix of derivative}, Journal of Information and Computational Science, 10:9 (2013), 2757-2764.


\bibitem{bio1} K. Holmaker, \emph{Global asymptotic stability for a stationary solution of a system of integro-differential equations describing the formation of liver zones}, SIAM J. Math. Anal., 24 (1) (1993), 116-128.

\bibitem{Hsiao} C.H. Hsiao, \emph {Hybrid function method for solving Fredholm and Volterra integral equations of the second kind}, J. Comput. Appl. Math., 230 (2009), 59-68.

\bibitem{el1} E.V. Kovalenko, \emph{Some approximate methods for solving integral equations of mixed problems}, Provl. Math. Mech., 53 (1) (1989), 85-92.

\bibitem{Malek-Basirat-Hashemi} K. Maleknejad, B. Basirat and E. Hashemizadeh, \emph{Hybrid Legendre polynomials and block-pulse functions approach for nonlinear Volterra-Fredholm integro-differential equations}, Comput. Math. Appl., 61(9) (2011), 282-288.

\bibitem{13H} H.R. Marzban and M. Razzaghi, \emph{Optimal control of linear delay systems via
hybrid of block-pulse and Legendre polynomials}, J. Franklin Inst., 341(3) (2004), 279-293.

\bibitem {Miza2} F. Mizaee, \emph{Numerical solution of system of linear integral equations via improvement of block-pulse functions}, Journal of Mathematical Modeling, 4 (2) (2016), 133-159.

\bibitem{Miza4} F. Mizaee, S. Alipour, \emph{Approximate solution of nonlinear quadratic integral equations of fractional order via piecewise linear functions}, Journal of Computational and Applied Mathematics, 331 (2018) 217-227.

\bibitem{Miza} F. Mizaee, S. Alipour and N. Samadyar, \emph{Numerical solution based on hybrid of block-pulse and parabolic functions for solving a system of nonlinear
stochastic Ito-Volterra integral equations of fractional order}, Journal of Computational and Applied Mathematics, 349 (2019),157-171.





\bibitem{Miza1} F. Mizaee and S. F. Hoseini, \emph{A new collocation approach for solving systems of high-order linear Volterra integro-differential equations with variable coefficients}, Applied Mathematics and Computation, 311 (2017) 72-282.


\bibitem{Miza3} F. Mizaee and S. F. Hoseini, \emph{Hybrid functions of Bernstein polynomials and block-pulse functions for solving optimal control of the nonlinear Volterra integral equations}, Indagationes Mathematicae, 27 (3) (2016), 835-849.

\bibitem{qu} A.D. Polyanin and A.V. Manzhirov, \emph{Handbook of integral equations} (2nd ed.), Chapman and Hall/CRC Press, Boca Raton- London (2008).

\bibitem{18R} M. Razzaghi and H.R. Marzban, \emph{A hybrid analysis direct method in the calculus of variations}, Int. J. Comput. Math., 75 (2000), 259-269.

\bibitem{el2} B.J. Semetanian, \emph{On an integral equation for axially symmetric problem in the case of an elastic body containing an inclusion}, J. Appl. Math. Mech., 55 (3) (1991), 371-375.

\bibitem{VIM1} X. Shang and D. Han, \emph{Application of the variational iteration method for solving $n^{th}$-order integro-differential equations}, Journal of Computational and Applied Mathematics, 234 (5) (2010), 1442-1447.
\bibitem{Sho-Abadi-Golpar} T. Shojaeizadeh, Z. Abadi, and E. Golpar Raboky, \emph{Hybrid functions approach for
solving Fredholm and Volterra integral equations}, J. Prime Res. Math., 5 (2009), 124-132.

\bibitem{er1} G. W. Stewart,  \emph{Afternotes on numerical analysis}, Vol. 49. SIAM, 1996.




\bibitem{neutron} A.M. Wazwaz, \emph{A first course in integral equations}, World Scientifics, Singapore (1997).

\bibitem{fm} J.R. Willis and S. Nemat-Nasser, \emph{Singular perturbation solution of a class of singular integral equations}, Quart. Appl. Math., XLVIII (4) (1990), 741-753.




\end{thebibliography}
\end{document}